\documentclass[11pt]{article}
\usepackage{indentfirst}
\usepackage{amsmath, amscd, amssymb, amsthm, verbatim,mathrsfs}

\allowdisplaybreaks

\newtheorem{theorem}{Theorem}[section]
\newtheorem{corollary}[theorem]{Corollary}
\newtheorem{lemma}[theorem]{Lemma}

\newtheorem{remark}[theorem]{Remark}

\begin{document}
\title{Deforming convex curves with constant anisotropic length}

\author{Zezhen Sun \footnote{School of Mathematical Sciences, East China Normal University, Shanghai 200241, China. E-mail address: \texttt{52205500017@stu.ecnu.edu.cn}}
}

\maketitle
\begin{abstract}
  \noindent
  In this paper, we study a curve flow which preserves the anisotropic length of the evolving
  curve, and show that for any convex closed initial curve, the flow exists for all time and the evolving curve converges to a homothety of the boundary of some Wulff shape defined by anisotropic function as time $t\to\infty$.
\end{abstract}

\section{Introduction}
The curvature flow of plane curves, arising in many application fields, such as phase transitions, crystal growth, image processing and smoothing, etc., has received a lot of attention in the last few decades. In general, the evolution equation has the form
\begin{equation}\label{flo}
  \left\{\begin{aligned}
    \frac{\partial X(u, t)}{\partial t}&=f(\kappa(u,t))\textbf{N}_{in}(u,t),\\
    X(u, 0)&=X_{0}(u),
  \end{aligned}
  \right.
  \end{equation}
where $X_{0}(u)\subset\mathbb{R}^{2}$ is a given smooth closed curve, parameterized by $u\in S^{1}$, and $X(u, t):S^{1}\times[0,T)\rightarrow\mathbb{R}^{2}$ is a family of curves moving along its inward normal direction $
\textbf{N}_{in}(u,t)$ with given speed function $f(\kappa(u,t))$, which is a strictly increasing (parabolicity) function
of the curvature $\kappa(u,t)$ of $X(u, t)$.

When $f(\kappa(u,t))=\kappa(u,t)$, \eqref{flo} is the well-known curve shortening flow, which has been
intensively studied by many scholars for various conditions on the initial curve $X_{0}$. One can see the book \cite{Chou-Zhu} for literature. In particular, we mention the papers \cite{Gage83,Gage85,Gage-Hamilton,Grayson,Grayson87}. Another class of interesting curvature flow is the so-called nonlocal curvature flow, such as the area-preserving flows \cite{Ma-Cheng,Gage85,Mao-Pan-Wang,Sun} and the length-preserving flows \cite{Tsai-Wang,Pan-Yang,Ma-Zhu} etc.

A number of mathematicians have begun to focus on the following question:

\textbf{Question 1}: Whether one can deform a simple curve $X_{0}$ into another one $\widetilde{X}$ by a curve flow?

Early in the year 1993, Gage-Li \cite{Gage93,Gage94} studied an anisotropic curve shortening flow
\begin{equation*}
  \left\{\begin{aligned}
    \frac{\partial X(u, t)}{\partial t}&=(\widetilde{p}/\widetilde{\kappa})\kappa \textbf{N}_{in}(u,t),\\
    X(u, 0)&=X_{0}(u),
  \end{aligned}
  \right.
  \end{equation*}
where $\widetilde{p}=-<\widetilde{X},\widetilde{N}>$ and $\widetilde{\kappa}$ are the support function and the curvature of the target curve $\widetilde{X}$, respectively. They proved that the evolving curve converges to the shape of the target curve as $X(\cdot,t)$ shrinks to a point, provided that the convex domain bounded by $\widetilde{X}$ is symmetric.

Since anisotropy is indispensable when working with crystalline materials, Mullin's theory was generalized by Angement-Gurtin\cite{an1,an2} and by Gurtin\cite{gu1,gu2} to include anisotropy in the motion of interface. Following their work, Chou-Zhu \cite{cz1,cz2} considered a generalized anisotropic flow
\begin{equation*}
  \left\{\begin{aligned}
    \frac{\partial X(u, t)}{\partial t}\beta(\theta)&=\big(g(\theta)\kappa+F\big) \textbf{N}_{in}(u,t),\\
    X(u, 0)&=X_{0}(u),
  \end{aligned}
  \right.
  \end{equation*}
where $\theta$ is the tangent angle of the interface, $\beta>0$ represents the resistance of the interfacial boundary to motion, $g$ is a positive function of the form $$g=d^{2}f/d\theta^{2}+f,$$
with $f>0$ being the interfacial energy, and the constant $F$ is the energy of the solid relative to its surrounding. They showed that the behavior of the solution can be classified according to a critical value $F^{\ast}$. One of the cases is when $F=F^{\ast}$, the flow exists globally and $X$ converges to a stationary solution. This an alternative model to evolve $X_{0}$ to $\widetilde{X}$.

In \cite{iv}, Ivaki deformed a convex curve into an ellipse by a centro-affine curvature flow. In \cite{py}, Pan-Yang deformed a convex curve into another centrally one by an anisotropic area-preserving curvature
flow. For other studies related to the \textbf{Question 1}, one can also refer to Bene$\breve{s}$-Yazaki-Kimura \cite{byk}, Lin-Tsai\cite{lt}, Gao-Zhang\cite{gz}, Pan-Zhang\cite{pz}, Yazaki\cite{ya}, Li-Wang\cite{lw} and etc.

Inspired by their works, we now investigate the following anisotropic curvature flow for convex curves
\begin{equation}\label{flow}
  \left\{\begin{aligned}
    \frac{\partial X(u, t)}{\partial t}&=\bigg(\widetilde{p}\big(\mathcal{K}-\frac{\lambda(t)}{2\widetilde{A}}\big)\bigg)\textbf{N}_{in}(u,t),\\
    X(u, 0)&=X_{0}(u),
  \end{aligned}
  \right.
  \end{equation}
where $\lambda(t)=\int_{X (\cdot, t)}\widetilde{p}\mathcal{K}^{2}ds$, $\mathcal{K}$ is the anisotropic curvature of $X$, $\widetilde{A}$ denotes the enclosed area of the target curve $\widetilde{X}$ and it is a constant. Under the flow \eqref{flow}, the anisotropic length of the evolving curve is invariable, which makes this flow different from all the previous ones. The detailed definitions of anisotropic curvature and anisotropic length can be found in next section.
\section{Main theorem}
For a convex curve $X$, we can parameterize it by its tangential angle $\theta\in S^{1}$. The relation
between arc-length $s$, tangential angle $\theta$ and curvature $\kappa$ of $X$ can be described by
$$d\theta=\kappa ds.$$
The support function $p(\theta)$ of $X$ is the distance from the origin to the tangent line
of this curve, that is,
$$p(\theta)=-<X,\textbf{N}>.$$
where $\textbf{N}$ is the unit inward pointing normal vector along this curve. The relation between the curvature $\kappa$ and the support function $p$ is
$$\kappa(\theta)=\frac{1}{p(\theta)+p^{''}(\theta)}.$$
Besides, the length of $X$ is equal to
$$L=\int^{L}_{0}ds=\int^{2\pi}_{0}\frac{d\theta}{\kappa}=\int^{2\pi}_{0}pd\theta,$$
and the enclosed area of $X$ is equal to
\begin{equation}\label{a}
A=\frac{1}{2}\int_{0}^{L}pds=\frac{1}{2}\int_{0}^{L}\frac{p}{\kappa}d\theta=\frac{1}{2}\int_{0}^{2\pi}(p^{2}-p^{2}_{\theta})d\theta.
\end{equation}
Let $\widetilde{p}$ be a given positive smooth function satisfying
\begin{equation}\label{condition}
\widetilde{p}(\theta+2\pi)=\widetilde{p}(\theta),\quad\widetilde{p}^{''}(\theta)+\widetilde{p}(\theta)>0,\quad\theta\in S^{1}
\end{equation}
Then we have a positive \textbf{anisotropic function} $\phi(\theta)$, i.e.,
$$\phi(\theta)=\widetilde{p}^{''}(\theta)+\widetilde{p}(\theta),\quad\theta\in S^{1}.$$
Assume
$$0<M_{1}\le\phi(\theta)\le M_{2},\quad\theta\in S^{1}.$$
where $M_{1}$ and $M_{2}$ are the minimum and maximum of function $\phi(\theta)$ on $S^{1}$, respectively. For a given smooth function $\widetilde{p}(\theta)>0$ satisfying \eqref{condition}, the \textbf{anisotropic curvature} of a convex curve $X$ is defined as (see \cite{ds})
$$\mathcal{K}=\phi(\theta)\kappa(\theta,t)=(\widetilde{p}^{''}(\theta)+\widetilde{p}(\theta))\kappa(\theta,t)$$
and the \textbf{anisotropic length}(or total interfacial energy) is defined as
\begin{equation}\label{anl}
\mathcal{L}=\int_{0}^{L}\widetilde{p}ds=\int^{2\pi}_{0}\frac{\widetilde{p}}{\kappa}d\theta.
\end{equation}
If $\widetilde{p}\equiv1$ then $\mathcal{L}$ is just the total length $L$ of curve $X$.
Furthermore, it is easy to see that the \textbf{anisotropic total curvature},
$$\int_{0}^{L}\mathcal{K}ds=\int_{0}^{L}(\widetilde{p}^{''}+\widetilde{p})\kappa ds=\int^{2\pi}_{0}\widetilde{p}d\theta=K_{0},$$
is a constant.

To state our main theorem, we now introduce the definition of Wulff shape(see \cite{ds,gu3}). For a given smooth function $\widetilde{p}(\theta)>0$ satisfying \eqref{condition}, the \textbf{Wulff shape} is defined as an intersection of halfspaces:
$$\widetilde{X}=\bigcap_{\theta\in S^{1}}\{\textbf{r}=(x_{1},x_{2})\,\,:\,\,-\textbf{r}\cdot\widetilde{\textbf{N}}\le \widetilde{p}(\theta)\}. $$
If the boundary $\partial\widetilde{X}$ of the Wulff shape $\widetilde{X}$ is smooth and it is parameterized by $\partial\widetilde{X}=\{\textbf{r}\,\,:\,\,\textbf{r}=-\widetilde{p}(\theta)\widetilde{\textbf{N}}+a(\theta)\widetilde{\textbf{T}},\,\theta\in S^{1}\}$, then it follows from the relation $d\theta=\widetilde{\kappa}ds$ that
$$\widetilde{\textbf{T}}=\frac{\partial\widetilde{X}}{\partial s}=(-\widetilde{p}^{'}(\theta)+a(\theta))\widetilde{\kappa}\widetilde{\textbf{N}}
+(\widetilde{p}(\theta)+a^{'}(\theta))\widetilde{\kappa}\widetilde{\textbf{T}}.$$
Hence $a(\theta)=\widetilde{p}^{'}(\theta)$ and $(\widetilde{p}^{''}(\theta)+\widetilde{p}(\theta))\widetilde{\kappa}=1$. Then the boundary $\partial\widetilde{X}$ can be parameterized as follows
$$\partial\widetilde{X}=\{\textbf{r}\,\,:\,\,\textbf{r}=-\widetilde{p}(\theta)\widetilde{\textbf{N}}
+\widetilde{p}^{'}(\theta)\widetilde{\textbf{T}},\,\theta\in S^{1}\},$$
and its support function is $\widetilde{p}(\theta)$. The curvature of $\partial\widetilde{X}$ is given by $\widetilde{\kappa}=(\widetilde{p}^{''}(\theta)+\widetilde{p}(\theta))^{-1}=\phi(\theta)^{-1}$. It means that the anisotropic curvature $\mathcal{K}$ of $\partial\widetilde{X}$ is a constant, i.e., $\mathcal{K}\equiv1$. Moreover, the area $\widetilde{A}$ of the Wulff shape $\widetilde{X}$ satisfies
\begin{equation}\label{wua}
\widetilde{A}=-\frac{1}{2}\int_{\partial\widetilde{X}}\textbf{r}\cdot\widetilde{\textbf{N}}ds=
\frac{1}{2}\int_{\partial\widetilde{X}}\widetilde{p}ds=\frac{1}{2}\int^{2\pi}_{0}
\frac{\widetilde{p}}{\widetilde{\kappa}}d\theta=\frac{1}{2}\int^{2\pi}_{0}(\widetilde{p}^{2}-\widetilde{p}^{2}_{\theta})d\theta.
\end{equation}
By \eqref{anl}, we also have $\widetilde{A}=\frac{1}{2}\mathcal{L}(\partial\widetilde{X})$. Obviously, $\widetilde{A}=\pi$ for the case $\widetilde{p}=1$. A Wulff shape $\widetilde{X}$ is called symmetric if
$$\widetilde{p}(\theta+\pi)=\widetilde{p}(\theta),\quad\theta\in S^{1}.$$
The main result of this paper is as follows.

\begin{theorem}\label{them}
Let $\widetilde{p}$ be a positive smooth $2\pi$-periodic function satisfying $\widetilde{p}_{\theta\theta}+\widetilde{p}>0$ and $\widetilde{p}(\theta+\pi)=\widetilde{p}(\theta)$. A smooth, closed, embedded, convex plane curve which evolves according to \eqref{flow} remains convex, preserves anisotropic length while increasing the area enclosed by the curve, and converges to a symmetric, convex curve that is a homothety of the boundary of Wulff shape $\widetilde{X}$ with support function $\widetilde{p}$ and curvature $\widetilde{\kappa}$ as time t goes to infinity in the $C^{\infty}$ metric.
\end{theorem}
\begin{remark}
If $\widetilde{p}=1$(i.e., the target curve $\widetilde{X}$ is a unit circle) then the evolution
equation \eqref{flow} will become a length-preserving curve flow
\begin{equation}\label{fl}
  \left\{\begin{aligned}
    \frac{\partial X(u, t)}{\partial t}&=\big(\kappa-\frac{1}{2\pi}\int^{L}_{0}\kappa^{2}ds\big)\textbf{N}_{in}(u,t),\\
    X(u, 0)&=X_{0}(u),
  \end{aligned}
  \right.
  \end{equation}
which has been studied by Ma-Zhu in \cite{Ma-Zhu}.
\end{remark}
We point out that the local existence and uniqueness of the flow \eqref{flow} on $S^{1}\times[0,T)$ for some short time $T>0$ follows in the same way as in Theorem 4.3 in \cite{py}. By now, this part is standard and we omit the detail. The main purpose of this paper is to study the long-time behavior of the flow \eqref{flow}.

This paper is organized as follows. In section 3, we give some preparation lemmas on the monotonicity of some geometric quantities and convexity-preserving property of evolving curves. In section 4, the long time existence of the evolving curve is proved. In section 5 is devoted to the proof of the flow's convergence.
\section{Preparations}
Let $g(u,t)=\left|X_{u}\right|=(x^{2}_{u}+y^{2}_{u})^{\frac{1}{2}}$, then the differential of arc-length is $ds = g(u,t)du$,
which can also be expressed as
$\frac{\partial}{\partial s}=\frac{1}{g}\frac{\partial}{\partial u}$ and $\frac{\partial s}{\partial u}=g.$
  The tangent vector $\textbf{T}$, normal vector $\textbf{N}$, direction angle $\theta$, curvature $\kappa$, perimeter $L$ and area $A$ can be defined as the following equations:
  \begin{align*}
    T&=\frac{\partial X}{\partial s} = \frac{1}{g}\frac{\partial X}{\partial u}, \quad \textbf{N} = \frac{1}{\kappa}\frac{\partial \textbf{T}}{\partial s} = \frac{1}{\kappa g}\frac{\partial \textbf{T}}{\partial u},\qquad\,\theta = \angle (\textbf{T},x),\quad
    \kappa =\frac{\partial \theta}{\partial s}= \frac{1}{g} \frac{\partial \theta}{\partial u},\\
      A(t)&=\frac{1}{2}\int^{L}_{0}\,xdy-ydx=-\frac{1}{2}\int^{L}_{0} \langle X,\textbf{N}\rangle ds, \quad
   L(t)=\int^{b}_{a} g(u,t)du = \int^{L}_{0} ds.
  \end{align*}
Since changing the tangential component of the velocity vector only affects the parameter representation of the curve, which does not affect the geometric shapes of the evolving curve \cite{Gage-Hamilton,Chou-Zhu}. We can choose a suitable tangential component $\xi=-\frac{\partial}{\partial\theta}\left(\widetilde{p}\big(\mathcal{K}-\frac{\lambda(t)}{2\widetilde{A}}\big)\right)$ to make $\theta$ independent of $t$, which will simplify calculations on the curve. We now consider the following evolution equation, which is equivalent to \eqref{flow}
\begin{equation}\label{flow1}
  \left\{\begin{aligned}
    \frac{\partial X(\theta, t)}{\partial t}&=\xi(\theta,t) T(\theta,t)+\bigg(\widetilde{p}\big(\mathcal{K}-\frac{\lambda(t)}{2\widetilde{A}}\big)\bigg)\textbf{N}_{in}(\theta,t),\\
    X(\theta, 0)&=X_{0}(\theta).
  \end{aligned}
  \right.
  \end{equation}

Just as Gage, Hamilton and others have done, we can derive the evolution equations of the enclosed area $A(t)$, the curvature $\kappa$ and the support function $p$ of the evolving curve as follows.
\begin{lemma}
Under the flow \eqref{flow1}, the geometric quantities evolve as
\begin{align}
  \label{at}
  \frac{d A}{d t}&=-\int^{L}_{0}(\widetilde{p}(\mathcal{K}-\frac{\lambda(t)}{2\widetilde{A}}))ds,\\
  \label{kappat}
  \frac{\partial\kappa}{\partial\,t}&=\kappa^{2}\bigg[\big(\widetilde{p}(\mathcal{K}-\frac{\lambda(t)}{2\widetilde{A}})\big)_{\theta\theta}+
  \widetilde{p}(\mathcal{K}-\frac{\lambda(t)}{2\widetilde{A}})\bigg],\\
  \label{pt}
 \frac{\partial p}{\partial\,t}&=-\widetilde{p}(\mathcal{K}-\frac{\lambda(t)}{2\widetilde{A}}).
\end{align}
\end{lemma}
Now the problem \eqref{flow1} could be reformulated through the evolution equation of curvature $\kappa(\theta,t)$
\begin{equation}\label{kt}
  \left\{\begin{aligned}
    \frac{\partial\kappa}{\partial t}=&\kappa^{2}\big[\widetilde{p}\mathcal{K}_{\theta\theta}+2\widetilde{p}_{\theta}\mathcal{K}_{\theta}+
    (\widetilde{p}_{\theta\theta}+\widetilde{p})(\mathcal{K}-\frac{\lambda(t)}{2\widetilde{A}})\big],\\
    \kappa(\theta, 0)&=\kappa_{0}(\theta)>0.
  \end{aligned}
  \right.
  \end{equation}
where $\kappa_{0}(\theta)$ is the curvature of $X_{0}$.
Since $\mathcal{K}=(\widetilde{p}_{\theta\theta}+\widetilde{p})\kappa=\phi\kappa,$ the problem \eqref{kt} could be written as
\begin{equation}\label{mkt}
  \left\{\begin{aligned}
    \frac{\partial\mathcal{K}}{\partial t}=&\frac{\mathcal{K}^{2}}{\phi}\big[\widetilde{p}\mathcal{K}_{\theta\theta}+2\widetilde{p}_{\theta}\mathcal{K}_{\theta}+
    (\widetilde{p}_{\theta\theta}+\widetilde{p})(\mathcal{K}-\frac{\lambda(t)}{2\widetilde{A}})\big],\\
    \mathcal{K}(\theta, 0)&=\mathcal{K}_{0}(\theta)>0,
  \end{aligned}
  \right.
  \end{equation}
\begin{lemma}\label{ala}
Under the flow\eqref{flow1}, the anisotropic length $\mathcal{L}(t)$ of evolving curve is preserved and the enclosed area $A(t)$ is non-decreasing during the evolution process.
\end{lemma}
\begin{proof}
By \eqref{kt}, \eqref{anl} and integration by parts, we have
\begin{align*}
\frac{d\mathcal{L}}{dt}&=\int^{2\pi}_{0}-\frac{\widetilde{p}}{\kappa^{2}}\kappa^{2}\big[\widetilde{p}
\mathcal{K}_{\theta\theta}+2\widetilde{p}_{\theta}\mathcal{K}_{\theta}+
    (\widetilde{p}_{\theta\theta}+\widetilde{p})(\mathcal{K}-\frac{\lambda(t)}{2\widetilde{A}})\big]d\theta\\
    &=-\int^{2\pi}_{0}\widetilde{p}^{2}\mathcal{K}_{\theta\theta}+2\widetilde{p}\widetilde{p}_{\theta}\mathcal{K}_{\theta}
    +\widetilde{p}(\widetilde{p}_{\theta\theta}+\widetilde{p})(\mathcal{K}-\frac{\lambda(t)}{2\widetilde{A}})d\theta\\
    &=-\int^{2\pi}_{0}\frac{\widetilde{p}\mathcal{K}}{\kappa}(\mathcal{K}-\frac{\lambda(t)}{2\widetilde{A}})d\theta\\
    &=-\int_{0}^{L}\widetilde{p}\mathcal{K}^{2}ds+\frac{\int_{0}^{L}\widetilde{p}\mathcal{K}ds}{2\widetilde{A}}\int_{0}^{L}\widetilde{p}\mathcal{K}^{2}ds.
\end{align*}
Notice that
\begin{equation}\label{ta}
\int_{0}^{L}\widetilde{p}\mathcal{K}ds=\int_{0}^{L}\widetilde{p}(\widetilde{p}_{\theta\theta}+\widetilde{p})\kappa ds=
\int^{2\pi}_{0}\widetilde{p}(\widetilde{p}_{\theta\theta}+\widetilde{p})d\theta=\int^{2\pi}_{0}
(\widetilde{p}^{2}-\widetilde{p}^{2}_{\theta})d\theta=2\widetilde{A}.
\end{equation}
Thus we have$\frac{d\mathcal{L}}{dt}=0$. Next, we recall a useful geometric inequality. Let $\widetilde{p}$ be a given positive smooth function satisfying conditions \eqref{condition} and $\widetilde{p}(\theta+\pi)=\widetilde{p}(\theta)$. Then for an embedded closed convex curve $X$, its anisotropic curvature $\mathcal{K}$, enclosed area $A$ and anisotropic length $\mathcal{L}$ satisfy(see \cite{Green-Osher})
\begin{equation}\label{wgi}
\int_{0}^{L}\widetilde{p}\mathcal{K}^{2}ds\ge\frac{\widetilde{A}\mathcal{L}}{A}\quad\text{(Wulff-Gage inequality)}.
\end{equation}
Here, we remark that a centro-symmetric condition $\widetilde{p}(\theta+\pi)=\widetilde{p}(\theta)$ is required for \eqref{wgi} to hold. By \eqref{at}, we calculate
\begin{align*}
\frac{dA}{dt}&=-\int_{0}^{L}\widetilde{p}\mathcal{K}ds+\frac{\int_{0}^{L}\widetilde{p}ds}{2\widetilde{A}}\int_{0}^{L}\widetilde{p}\mathcal{K}^{2}ds\\
&=-2\widetilde{A}+\frac{\mathcal{L}}{2\widetilde{A}}\int_{0}^{L}\widetilde{p}\mathcal{K}^{2}ds\\
&\ge\frac{\mathcal{L}^{2}}{2A}-2\widetilde{A}=\frac{\mathcal{L}^{2}-4\widetilde{A}A}{2A}\ge0.
\end{align*}
where we have used the Minkowski inequality(see \cite{Gage93})
\begin{equation}\label{mi}
\mathcal{L}^{2}\ge4\widetilde{A}A
\end{equation}
in the last inequality.
\end{proof}
\begin{remark}(The lower bound of $\lambda(t)$)
By \eqref{wgi} and \eqref{mi}, we can get
$$\int_{0}^{L}\widetilde{p}\mathcal{K}^{2}ds\ge\widetilde{A}\mathcal{L}\cdot\frac{4\widetilde{A}}{\mathcal{L}^{2}}=
\frac{4\widetilde{A}^{2}}{\mathcal{L}}=\frac{4\widetilde{A}^{2}}{\mathcal{L}(0)}>0.$$
\end{remark}

\begin{lemma}\label{conv}
A strictly convex curve evolving according to \eqref{flow1} remains so during the evolution process.
\end{lemma}
\begin{proof}
Suppose the flow exists in time interval $[0,T)$ and there is a smallest $t_{1}\in(0,T)$ such that $\kappa$ touches zero at $t_{1}$, that is,
$$t_{1}=\inf\{t:\kappa(\theta,t)=0\,\,\text{for some}\,\,\theta\in S^{1}\}.$$
We rewrite the equation \eqref{mkt} of $\mathcal{K}$ as
$$\mathcal{K}_{t}=a(\theta,t)\mathcal{K}_{\theta\theta}+b(\theta,t)\mathcal{K},\quad(\theta,t)\in S^{1}\times(0,t_{1}),$$
where $a(\theta,t)=\frac{\widetilde{p}\mathcal{K}^{2}}{\phi}$ and $b(\theta,t)=\frac{\mathcal{K}}{\phi}[2\widetilde{p}_{\theta}\mathcal{K}_{\theta}+
    (\widetilde{p}_{\theta\theta}+\widetilde{p})(\mathcal{K}-\frac{\lambda(t)}{2\widetilde{A}})]$. Since $t_{1}<T,$ there exists a constant $C(t_{1})$ such that $0\le\mathcal{K}(\theta,t)\le C(t_{1})$ on $S^{1}\times(0,t_{1})$. Combining this with the bound estimate for $\mathcal{K}_{\theta}$ (see Lemma \ref{graest}), we can find a constant $D(t_{1})$ such that $\lvert b(\theta,t)\lvert\le D(t_{1})$ on $S^{1}\times(0,t_{1})$.

    Set $U(\theta,t)=\mathcal{K}(\theta,t)e^{Dt}$. Using the maximum principle argument to the equation of
    $U(\theta,t)$ on $S^{1}\times[0,t_{1})$, one can obtain $\min_{S^{1}}U(\theta,t)\ge\min_{S^{1}}U(\theta,0)$ for $t\in(0,t_{1})$. As a result,
    $$\min_{S^{1}}\mathcal{K}(\theta,t_{1})\ge
    e^{-Dt}\min_{S^{1}}U(\theta,0)=e^{-Dt}\min_{S^{1}}\mathcal{K}(\theta,0).$$
    This is a contraction. Thus we have $\kappa(\theta,t)>0$ on $S^{1}\times[0,T)$.
\end{proof}

\section{Global existence of the flow}
\subsection{Uniform upper bound of the curvature; Tso's method}
In this section, we shall use the support function method to derive a time-independent bound of the curvature.
Given a convex curve $X$ in the plane, the Bonnesen inequality (\cite{Osserman-Bonnesn}) says that
\begin{equation*}
\rho L-A-\pi\rho^{2}\ge0,\quad r_{in}\le\rho\le r_{out},
\end{equation*}
where $r_{in}$ and $r_{out}$ are the inradius and outradius of the domain bounded by $X$, respectively. So
\begin{equation}\label{r1}
0<\frac{L-\sqrt{L^{2}-4\pi A}}{2\pi}\le r_{in}\le r_{out}\le\frac{L+\sqrt{L^{2}-4\pi A}}{2\pi}.
\end{equation}
Since the anisotropic length $\mathcal{L}(t)$ is preserved in time, we have
\begin{equation}\label{r2}
L+\sqrt{L^{2}-4\pi A}\le2L\le M_{3}\int_{0}^{L}\widetilde{p}ds=M_{3}\mathcal{L}(t)=M_{3}\mathcal{L}(0),
\end{equation}
where $M_{3}$ are the maximum of functions $\frac{2}{\widetilde{p}(\theta)}$. Combining this with the fact that the area $A(t)$ is increasing, we get
\begin{equation}\label{r3}
L-\sqrt{L^{2}-4\pi A}=\frac{4\pi A}{L+\sqrt{L^{2}-4\pi A}}\ge\frac{4\pi A(0)}{M_{3}\mathcal{L}(0)}.
\end{equation}
By \eqref{r1}, \eqref{r2} and \eqref{r3}, we have
\begin{equation}
0\le\frac{2A(0)}{M_{3}\mathcal{L}(0)}\le r_{in}\le r_{out}\le\frac{M_{3}\mathcal{L}(0)}{2\pi}.
\end{equation}
which implies that both the inradius and the outradius have time-independent positive bounds. Now we can show that the anisotropic curvature of the evolving curve has an upper bound independent of time.

\begin{lemma}\label{kupj}
Under the flow \eqref{flow1}, there exists a constant $C_{0}$ independent of time such that
\begin{equation}\label{kfij}
0<\mathcal{K}(\theta,t)\le C_{0},\quad\forall(\theta,t)\in S^{1}\times[0,T).
\end{equation}
\end{lemma}
\begin{proof}
Let $B(0)$ be the inscribed circle of $X(\theta, 0)$ with radius $r(0)=r_{in}(0)$. We shrink $B(0)$ by the contraction flow
\begin{equation}
\frac{\partial B}{\partial t}=M_{4}\kappa N_{in},
\end{equation}
where $M_{4}$ is the maximum of $\widetilde{p}(\theta)\phi(\theta)$.  The radius $r(t)$ of
$B(t)$ is given by
\begin{equation}
r(t)=\sqrt{r^{2}(0)-2M_{4}t},\quad t\in[0,min\{r^{2}(0)/2M_{4},T\}).
\end{equation}
By the maximum principle, $B(t)$ is enclosed by $X(\theta, t)$.
If the center of $B(0)$ is chosen to be the origin $O$, then the support function $p(\theta, t)$ (with respective to $O$) of $X(\theta, t)$ satisfies
\begin{equation}\label{pj}
p(\theta, t)\ge\sqrt{r^{2}(0)-2M_{4}t},\quad (\theta,t)\in S^{1}\times[0,min\{r^{2}(0)/2M_{4},T\}).
\end{equation}
where $r(0)\ge\frac{2A(0)}{M_{3}\mathcal{L}(0)}$. Set $T_{0}=\frac{1}{2}min\{r^{2}(0)/2M_{4},T\}$. By \eqref{pj}, we have $p(\theta, t)\ge2\beta$ on $S^{1}\times[0,T_{0}]$, where $\beta>0$ is a constant only depending on the initial curve $X(\theta, 0)$, given by
\begin{equation}
\beta=\frac{\sqrt{2}A(0)}{2M_{3}\mathcal{L}(0)}.
\end{equation}
Moreover, since the anisotropic length of evolving curve is preserved, one has an upper bound of the support function
\begin{equation}
p(\theta, t)\le\frac{L(t)}{2}\le\frac{1}{2}(\min_{S^{1}}\widetilde{p}(\theta))^{-1}\mathcal{L}(0).
\end{equation}
So there is a constant $D_{0}>0$ depending only on $X(\theta, 0)$ such that $p(\theta, t)\le D_{0}$ on $S^{1}\times[0,T_{0}]$. Hence
\begin{equation}
0<2\beta\le p(\theta, t)\le D_{0}\quad\text{on}\quad S^{1}\times[0,T_{0}].
\end{equation}
Consider the evolution of the quantity
\begin{equation}
Q(\theta, t)=\frac{\widetilde{p}(\theta,t)\mathcal{K}(\theta,t)}{p(\theta, t)-\beta},\quad(\theta,t)\in S^{1}\times[0,T_{0}].
\end{equation}
The first and the second derivatives of $Q$ with respect to $\theta$ are
$$Q_{\theta}=\frac{(\widetilde{p}\mathcal{K})_{\theta}}{p-\beta}-\frac{p_{\theta}\widetilde{p}\mathcal{K}}{(p-\beta)^{2}}$$
and
$$Q_{\theta\theta}=\frac{(\widetilde{p}\mathcal{K})_{\theta\theta}}{p-\beta}-\frac{2p_{\theta}(\widetilde{p}\mathcal{K})_{\theta}}{(p-\beta)^{2}}
-\frac{p_{\theta\theta}\widetilde{p}\mathcal{K}}{(p-\beta)^{2}}+\frac{2p^{2}_{\theta}\widetilde{p}\mathcal{K}}{(p-\beta)^{3}}.$$
Combining with the evolution \eqref{pt} of the support function $p$, we obtain
\begin{align}\label{Qt}
Q_{t}=&\frac{\widetilde{p}\mathcal{K}^{2}}{\phi(p-\beta)}\big[(\widetilde{p}\mathcal{K})_{\theta\theta}+\widetilde{p}\mathcal{K}
-\frac{\lambda(t)\phi}{2\widetilde{A}}\big]+\frac{\widetilde{p}^{2}\mathcal{K}(\mathcal{K}-\frac{\lambda(t)}{2\widetilde{A}})}{(p-\beta)^{2}}\notag\\
=&\frac{\widetilde{p}\mathcal{K}^{2}}{\phi}\bigg[Q_{\theta\theta}+\frac{2p_{\theta}}{p-\beta}\big(Q_{\theta}+
\frac{p_{\theta}\widetilde{p}\mathcal{K}}{(p-\beta)^{2}}\big)+\frac{p_{\theta\theta}\widetilde{p}\mathcal{K}}{(p-\beta)^{2}}
-\frac{2p^{2}_{\theta}\widetilde{p}\mathcal{K}}{(p-\beta)^{3}}\bigg]\notag\\
&-\frac{\lambda(t)\widetilde{p}\mathcal{K}^{2}}
{2\widetilde{A}(p-\beta)}+\frac{\widetilde{p}^{2}\mathcal{K}^{3}}{\phi(p-\beta)}+
\frac{\widetilde{p}^{2}\mathcal{K}(\mathcal{K}-\frac{\lambda(t)}{2\widetilde{A}})}{(p-\beta)^{2}}\notag\\
=&\frac{\widetilde{p}\mathcal{K}^{2}}{\phi}Q_{\theta\theta}+\frac{2p_{\theta}\widetilde{p}\mathcal{K}^{2}}{\phi(p-\beta)}
Q_{\theta}+\frac{p_{\theta\theta}\widetilde{p}^{2}\mathcal{K}^{3}}{\phi(p-\beta)^{2}}-\frac{\lambda(t)\widetilde{p}\mathcal{K}^{2}}
{2\widetilde{A}(p-\beta)}+\frac{\widetilde{p}^{2}\mathcal{K}^{3}}{\phi(p-\beta)}+
\frac{\widetilde{p}^{2}\mathcal{K}(\mathcal{K}-\frac{\lambda(t)}{2\widetilde{A}})}{(p-\beta)^{2}}\notag\\
=&\frac{\widetilde{p}\mathcal{K}^{2}}{\phi}Q_{\theta\theta}+\frac{2p_{\theta}\widetilde{p}\mathcal{K}^{2}}{\phi(p-\beta)}
Q_{\theta}+\frac{\widetilde{p}^{2}\mathcal{K}^{3}}{\kappa\phi(p-\beta)^{2}}-
\frac{\beta\widetilde{p}^{2}\mathcal{K}^{3}}{\phi(p-\beta)^{2}}-\frac{\lambda(t)\widetilde{p}\mathcal{K}^{2}}
{2\widetilde{A}(p-\beta)}+\frac{\widetilde{p}^{2}\mathcal{K}^{}}{(p-\beta)^{2}}-\frac{\widetilde{p}^{2}\mathcal{K}
\lambda(t)}{2\widetilde{A}(p-\beta)^{2}}\notag\\
\le&\frac{\widetilde{p}\mathcal{K}^{2}}{\phi}Q_{\theta\theta}+\frac{2p_{\theta}\widetilde{p}\mathcal{K}^{2}}{\phi(p-\beta)}
Q_{\theta}+\frac{2\widetilde{p}^{2}\mathcal{K}^{2}}{(p-\beta)^{2}}-\frac{\beta\widetilde{p}^{2}\mathcal{K}^{3}}{\phi(p-\beta)^{2}}\notag\\
\le&\frac{\widetilde{p}\mathcal{K}^{2}}{\phi}Q_{\theta\theta}+\frac{2p_{\theta}\widetilde{p}\mathcal{K}^{2}}{\phi(p-\beta)}
Q_{\theta}+Q^{2}\big(2-\frac{\beta^{2}}{\widetilde{p}\phi}Q\big),
\end{align}
where we have used the familiar identity $p_{\theta\theta}+p=1/\kappa$ and the inequality $p-\beta\ge\beta$ on $S^{1}\times[0,T_{0}]$ in the above derivation. Also note that in the above estimate we have thrown away two terms containing $\lambda(t)$ due to the negative sign in front of them.

Let
 $$Q(\theta_{0},t_{0})=\max_{S^{1}\times[0,T_{0}]}Q(\theta,t)$$
Suppose $t_{0}\ne0$. At the point $(\theta_{0},t_{0})$, one has
$$Q_{t}\ge0,\quad Q_{\theta}=0\quad\text{and}\quad Q_{\theta\theta}\le0.$$
Thus taking in account with \eqref{Qt} we derive to obtain
$$Q(\theta_{0},t_{0})\le\frac{2\widetilde{p}\phi}{\beta^{2}}\le\frac{2M_{4}}{\beta^{2}}.$$
So we have
$$\max_{S^{1}\times[0,T_{0}]}Q\le\max\big\{\frac{2M_{4}}{\beta^{2}},\max_{S^{1}}Q(\theta,0)\big\}:=D_{1}$$
and
$$\max_{S^{1}\times[0,T_{0}]}\mathcal{K}\le D_{1}(p-\beta)(\min_{S^{1}}\widetilde{p})^{-1}\le
D_{1}(D_{0}-\beta)(\min_{S^{1}}\widetilde{p})^{-1}:=D_{2}.$$
The proof is done.
\end{proof}

\begin{remark}(The upper bound of $\lambda(t)$)
By \eqref{ta} and \eqref{kfij}, we have
$$\int_{0}^{L}\widetilde{p}\mathcal{K}^{2}ds\le C_{0}\int_{0}^{L}\widetilde{p}\mathcal{K}ds=2C_{0}\widetilde{A}.$$
\end{remark}
\subsection{Lower bound of the curvature}
In this section, we first establish the gradient estimate and then deduce the time-independent lower bound for the anisotropic curvature via Bernstein type estimates. The idea is from \cite{LT}.

\begin{lemma}\label{graest}(Gradient estimate)
Under the the flow \eqref{flow1}. there exists a time-independent constant $C_{1}$ such that
$$\lvert\mathcal{K}_{\theta}(\theta,t)\lvert\le C_{1},\quad\forall(\theta,t)\in S^{1}\times[0,T).$$
\end{lemma}
\begin{proof}
Let $F=\mathcal{K}_{\theta}+\alpha\mathcal{K}$ with any fixed $\alpha>0$. We first compute
\begin{align}\label{k1}
(\mathcal{K}_{\theta})_{t}=&\bigg\{\frac{\mathcal{K}^{2}}{\phi}\big[\widetilde{p}\mathcal{K}_{\theta\theta}+2\widetilde{p}_{\theta}\mathcal{K}_{\theta}+
    (\widetilde{p}_{\theta\theta}+\widetilde{p})(\mathcal{K}-\frac{\lambda(t)}{2\widetilde{A}})\big]\bigg\}_{\theta}\notag\\
=&\big(\frac{\mathcal{K}^{2}}{\phi}\big)_{\theta}\big(\widetilde{p}\mathcal{K}_{\theta\theta}+2\widetilde{p}_{\theta}\mathcal{K}_{\theta}\big)+
\big(\frac{2\mathcal{K}\mathcal{K}_{\theta}}{\phi}-\frac{\phi_{\theta}\mathcal{K}^{2}}{\phi}\big)\phi
\big(\mathcal{K}-\frac{\lambda(t)}{2\widetilde{A}}\big)\notag\\
&+\frac{\mathcal{K}^{2}}{\phi}\bigg[\widetilde{p}_{\theta}\mathcal{K}_{\theta\theta}+
+\widetilde{p}\mathcal{K}_{\theta\theta\theta}+2\widetilde{p}_{\theta\theta}\mathcal{K}_{\theta}+2\widetilde{p}_{\theta}\mathcal{K}_{\theta\theta}
+(\widetilde{p}_{\theta\theta\theta}+\widetilde{p}_{\theta})\big(\mathcal{K}-\frac{\lambda(t)}{2\widetilde{A}}\big)+
(\widetilde{p}_{\theta\theta}+\widetilde{p})\mathcal{K}_{\theta}\bigg]\notag\notag\\
=&\frac{\widetilde{p}\mathcal{K}^{2}}{\phi}\mathcal{K}_{\theta\theta\theta}+\bigg(\frac{3\mathcal{K}^{2}\widetilde{p}_{\theta}}{\phi}+
\widetilde{p}\big(\frac{\mathcal{K}^{2}}{\phi}\big)_{\theta}\bigg)\mathcal{K}_{\theta\theta}+\bigg[\frac{\mathcal{K}^{2}}{\phi}
(\widetilde{p}+3\widetilde{p}_{\theta\theta})+2\widetilde{p}_{\theta}\big(\frac{\mathcal{K}^{2}}{\phi}\big)_{\theta}
+2\mathcal{K}(\mathcal{K}-\frac{\lambda(t)}{2\widetilde{A}})\bigg]\mathcal{K}_{\theta}\notag\\
&+\bigg[\frac{\mathcal{K}^{2}}{\phi}(\widetilde{p}_{\theta\theta\theta}+\widetilde{p}_{\theta})-\frac{\phi_{\theta}
\mathcal{K}^{2}}{\phi}\bigg]\big(\mathcal{K}-\frac{\lambda(t)}{2\widetilde{A}}\big),
\end{align}
and by $F_{\theta}=\mathcal{K}_{\theta\theta}+\alpha\mathcal{K}_{\theta},\,F_{\theta\theta}=\mathcal{K}_{\theta\theta\theta}+\alpha\mathcal{K}_{\theta\theta}$, we get
$$F_{t}=\frac{\widetilde{p}\mathcal{K}^{2}}{\phi}F_{\theta\theta}+\bigg[\frac{3\mathcal{K}^{2}\widetilde{p}_{\theta}}{\phi}+
\widetilde{p}\big(\frac{\mathcal{K}^{2}}{\phi}\big)_{\theta}\bigg]F_{\theta}+G(\theta,t),$$
where
\begin{align*}
G(\theta,t)&=\big(\frac{4\widetilde{p}_{\theta}\mathcal{K}}{\phi}-\frac{2\alpha\widetilde{p}\mathcal{K}}{\phi}\big)\mathcal{K}_{\theta}^{2}
+\bigg[\frac{\mathcal{K}^{2}}{\phi}(\widetilde{p}_{\theta\theta\theta}+\widetilde{p}_{\theta})-\frac{\phi_{\theta}
\mathcal{K}^{2}}{\phi}+\alpha\mathcal{K}^{2}\bigg]\big(\mathcal{K}-\frac{\lambda(t)}{2\widetilde{A}}\big)\\
+&\bigg[\frac{\alpha\widetilde{p}\phi_{\theta}\mathcal{K}^{2}}{\phi^{2}}-\frac{\alpha\widetilde{p}_{\theta}\mathcal{K}^{2}}{\phi}
-\frac{2\widetilde{p}_{\theta}\phi_{\theta}\mathcal{K}^{2}}{\phi^{2}}+(3\widetilde{p}_{\theta\theta}+\widetilde{p})
\frac{\mathcal{K}^{2}}{\phi}+2\mathcal{K}(\mathcal{K}-\frac{\lambda(t)}{2\widetilde{A}})\bigg]\mathcal{K}_{\theta}
\end{align*}
Choose $\alpha$ sufficiently large such that the coefficient of $\mathcal{K}_{\theta}^{2}$ satisfies $\frac{4\widetilde{p}_{\theta}\mathcal{K}}{\phi}-\frac{2\alpha\widetilde{p}\mathcal{K}}{\phi}\le -D_{0}$ for some time-independent constant $D_{0}>0$. Without loss of generality, we can assume that $\mathcal{K}_{\theta}>0$. Since $\mathcal{K}$ and $\lambda(t)$ are bounded, we can find time-independent constants $D_{1}$ and $D_{2}$ such that
$$G(\theta,t)\le-D_{0}\mathcal{K}_{\theta}^{2}+D_{1}\mathcal{K}_{\theta}+D_{2},\quad\forall(\theta,t)\in S^{1}\times[0,T).$$
Thus, $G<0$ on the domain where $\mathcal{K}_{\theta}$ is large enough, or equivalently $F$ is large enough. Using maximum principle, one can obtain the upper bound for $F$ and thus for $\mathcal{K}_{\theta}$. Similarly, the upper bound for $-\mathcal{K}_{\theta}$ can be derived by setting $F=-\mathcal{K}_{\theta}+\alpha\mathcal{K}$ with any fixed $\alpha>0$. Thus, $\mathcal{K}_{\theta}$ is uniformly bounded on $S^{1}\times[0,T)$.
\end{proof}

\begin{lemma}\label{kloj}
Under the flow \eqref{flow1}, there exists a time-independent constant $C_{2}$ such that
$$\mathcal{K}(\theta,t)\ge C_{2},\quad\forall(\theta,t)\in S^{1}\times[0,T).$$
\end{lemma}
\begin{proof}
For all $t\in[0,T)$ and any $\theta_{1}, \theta_{2}\in S^{1}$, we have
\begin{align*}
\log\mathcal{K}_{\max}(t)-\log\mathcal{K}_{\min}(t)&=\int^{\theta_{2}}_{\theta_{1}}\frac{\mathcal{K}_{\theta}(\theta,t)}{\mathcal{K}(\theta,t)}d\theta\\
&\le\int^{2\pi}_{0}\frac{\big\lvert\mathcal{K}_{\theta}(\theta,t)\big\lvert}{\mathcal{K}(\theta,t)}d\theta\\
&\le C_{1}\int^{2\pi}_{0}\frac{1}{\mathcal{K}(\theta,t)}d\theta\\
&\le\frac{C_{1}}{\min_{S^{1}}(\widetilde{p}\phi)}
\int^{2\pi}_{0}\frac{\widetilde{p}}{\kappa}d\theta\\
&=\frac{C_{1}\mathcal{L}(t)}{\min_{S^{1}}(\widetilde{p}\phi)}=\frac{C_{1}\mathcal{L}(0)}{\min_{S^{1}}(\widetilde{p}\phi)}:=D_{0}.
\end{align*}
Then we have
$$\frac{\mathcal{K}_{\max}(t)}{\mathcal{K}_{\min}(t)}\le e^{D_{0}}.$$
Also by
$$\frac{2\pi\min_{S^{1}}(\widetilde{p}\phi)}{\mathcal{K}_{\max}}\le\int^{2\pi}_{0}\frac{\widetilde{p}\phi}{\mathcal{K}}d\theta
=\int^{2\pi}_{0}\frac{\widetilde{p}}{\kappa}d\theta=\mathcal{L}(t)=\mathcal{L}(0),$$
we have
$$\mathcal{K}_{\min}(t)\ge\mathcal{K}_{\max}(t)e^{-D_{0}}\ge\frac{2\pi\min_{S^{1}}(\widetilde{p}\phi)}{\mathcal{L}(0)}e^{-D_{0}}.$$
This finishes the proof.
\end{proof}

\subsection{Higher order estimates on curvature}
From Lemmas \ref{kupj} and \ref{kloj}, $\mathcal{K}$ has time-independent positive upper and lower bound.   Parabolic regularity theory can guarantee that all space-time derivatives of $\mathcal{K}$ remain bounded on $S^{1}\times[0,T)$. We now show this detail by the maximum principle argument.
\begin{lemma}\label{hies}
Under the flow \eqref{flow1}, there exists a constant $C_{3}$ independent of time such that
$$\lvert\mathcal{K}_{\theta\theta}(\theta,t)\lvert\le C_{3},\quad\forall(\theta,t)\in S^{1}\times[0,T).$$
\end{lemma}
\begin{proof}
Let $F=\mathcal{K}_{\theta\theta}+\alpha\mathcal{K}_{\theta}^{2}$, where $\alpha$ ia s constant to be chosen later on. We rewrite the equation \eqref{k1} as
\begin{equation}\label{k2}
(\mathcal{K}_{\theta})_{t}=\frac{\widetilde{p}\mathcal{K}^{2}}{\phi}\mathcal{K}_{\theta\theta\theta}+\bigg[\big(
\frac{\widetilde{p}\mathcal{K}^{2}}{\phi}\big)_{\theta}+\frac{2\widetilde{p}_{\theta}\mathcal{K}^{2}}{\phi}\bigg]\mathcal{K}_{\theta\theta}
+\bigg[2\big(\frac{\widetilde{p}_{\theta}\mathcal{K}^{2}}{\phi}\big)_{\theta}+3\mathcal{K}^{2}-\frac{\lambda(t)\mathcal{K}
}{\widetilde{A}}\bigg]\mathcal{K}_{\theta}.
\end{equation}
Furthermore, we have
\begin{align*}
(\mathcal{K}_{\theta\theta})_{t}=&\frac{\widetilde{p}\mathcal{K}^{2}}{\phi}\mathcal{K}_{\theta\theta\theta\theta}
+\bigg[2\big(\frac{\widetilde{p}\mathcal{K}^{2}}{\phi}\big)_{\theta}+\frac{2\widetilde{p}_{\theta}\mathcal{K}^{2}}{\phi}\bigg]
\mathcal{K}_{\theta\theta\theta}+\bigg[\big(\frac{\widetilde{p}\mathcal{K}^{2}}{\phi}\big)_{\theta\theta}+4\big(
\frac{\widetilde{p}_{\theta}\mathcal{K}^{2}}{\phi}\big)_{\theta}+3\mathcal{K}^{2}-\frac{\lambda(t)\mathcal{K}}{\widetilde{A}}\bigg]\mathcal{K}_{\theta\theta}\\
&+\bigg[2\big(\frac{\widetilde{p}_{\theta}\mathcal{K}^{2}}{\phi}\big)_{\theta\theta}+6\mathcal{K}\mathcal{K}_{\theta}
-\frac{\lambda(t)\mathcal{K}_{\theta}}{\widetilde{A}}\bigg]\mathcal{K}_{\theta}.
\end{align*}
Then we have
\begin{align*}
F_{t}=&\frac{\widetilde{p}\mathcal{K}^{2}}{\phi}\mathcal{K}_{\theta\theta\theta\theta}+\bigg[2\big(\frac{\widetilde{p}\mathcal{K}^{2}}{\phi}\big)_{\theta}+\frac{2\widetilde{p}_{\theta}\mathcal{K}^{2}}{\phi}
+\frac{2\alpha\widetilde{p}\mathcal{K}_{\theta}\mathcal{K}^{2}}{\phi}\bigg]\mathcal{K}_{\theta\theta\theta}\\
&+\bigg[\big(\frac{\widetilde{p}\mathcal{K}^{2}}{\phi}\big)_{\theta\theta}+4\big(
\frac{\widetilde{p}_{\theta}\mathcal{K}^{2}}{\phi}\big)_{\theta}+
2\alpha\mathcal{K}_{\theta}\big(\frac{\widetilde{p}\mathcal{K}^{2}}{\phi}\big)_{\theta}+\frac{4\alpha\widetilde{p}_{\theta}\mathcal{K}_{\theta}\mathcal{K}^{2}}{\phi}
+3\mathcal{K}^{2}-\frac{\lambda(t)\mathcal{K}}{\widetilde{A}}\bigg]\mathcal{K}_{\theta\theta}\\
&+\bigg[4\alpha\big(\frac{\widetilde{p}_{\theta}\mathcal{K}^{2}}{\phi}\big)_{\theta}+6\alpha\mathcal{K}^{2}+6\mathcal{K}-\frac{2\alpha\lambda(t)\mathcal{K}}{\widetilde{A}}-\frac{\lambda(t)}{\widetilde{A}}\bigg]
\mathcal{K}_{\theta}^{2}+2\big(\frac{\widetilde{p}_{\theta}\mathcal{K}^{2}}{\phi}\big)_{\theta\theta}\mathcal{K}_{\theta}\\
=&\frac{\widetilde{p}\mathcal{K}^{2}}{\phi}F_{\theta\theta}+\bigg[2\big(\frac{\widetilde{p}\mathcal{K}^{2}}{\phi}\big)_{\theta}+\frac{2\widetilde{p}_{\theta}\mathcal{K}^{2}}{\phi}\bigg]F_{\theta}
+G(\theta,t),
\end{align*}
where
\begin{align*}
G(\theta,t)=&\frac{-2\alpha\widetilde{p}\mathcal{K}^{2}\mathcal{K}_{\theta\theta}^{2}}{\phi}+\bigg[\big(\frac{\widetilde{p}\mathcal{K}^{2}}{\phi}\big)_{\theta\theta}+4\big(
\frac{\widetilde{p}_{\theta}\mathcal{K}^{2}}{\phi}\big)_{\theta}-
2\alpha\mathcal{K}_{\theta}\big(\frac{\widetilde{p}\mathcal{K}^{2}}{\phi}\big)_{\theta}
+3\mathcal{K}^{2}-\frac{\lambda(t)\mathcal{K}}{\widetilde{A}}\bigg]\mathcal{K}_{\theta\theta}\\
&+\bigg[4\alpha\big(\frac{\widetilde{p}_{\theta}\mathcal{K}^{2}}{\phi}\big)_{\theta}+6\alpha\mathcal{K}^{2}+6\mathcal{K}-\frac{2\alpha\lambda(t)\mathcal{K}}{\widetilde{A}}-\frac{\lambda(t)}{\widetilde{A}}\bigg]
\mathcal{K}_{\theta}^{2}+2\big(\frac{\widetilde{p}_{\theta}\mathcal{K}^{2}}{\phi}\big)_{\theta\theta}\mathcal{K}_{\theta}\\
=&A_{1}(\theta,t)\mathcal{K}_{\theta\theta}^{2}+A_{2}(\theta,t)\mathcal{K}_{\theta\theta}+A_{3}(\theta,t).
\end{align*}
Here,
$$A_{1}(\theta,t)=\frac{-2\alpha\widetilde{p}\mathcal{K}^{2}}{\phi}+\frac{2\widetilde{p}\mathcal{K}}{\phi},$$
$$A_{2}(\theta,t)=\big(\frac{\widetilde{p}\mathcal{K}^{2}}{\phi}\big)_{\theta\theta}+4\big(
\frac{\widetilde{p}_{\theta}\mathcal{K}^{2}}{\phi}\big)_{\theta}-
2\alpha\mathcal{K}_{\theta}\big(\frac{\widetilde{p}\mathcal{K}^{2}}{\phi}\big)_{\theta}
+3\mathcal{K}^{2}-\frac{\lambda(t)\mathcal{K}}{\widetilde{A}}+\frac{4\widetilde{p}_{\theta}\mathcal{K}_{\theta}\mathcal{K}}{\phi},$$
and
$$A_{3}(\theta,t)=\bigg[4\alpha\big(\frac{\widetilde{p}_{\theta}\mathcal{K}^{2}}{\phi}\big)_{\theta}+6\alpha\mathcal{K}^{2}+6\mathcal{K}-\frac{2\alpha\lambda(t)\mathcal{K}}{\widetilde{A}}-\frac{\lambda(t)}{\widetilde{A}}\bigg]
\mathcal{K}_{\theta}^{2}+2\big(\frac{\widetilde{p}_{\theta}\mathcal{K}^{2}}{\phi}\big)_{\theta\theta}\mathcal{K}_{\theta}-
\frac{4\widetilde{p}_{\theta}\mathcal{K}_{\theta}\mathcal{K}_{\theta\theta}\mathcal{K}}{\phi}.$$
Notice that $\mathcal{K},\mathcal{K}_{\theta}$ and $\lambda(t)$ are all bounded quantities. Choose $\alpha$ sufficiently large such that $A_{1}(\theta,t)$, the coefficient of $\mathcal{K}_{\theta\theta}^{2}$, is negative. Since the coefficients $A_{1}(\theta,t), A_{2}(\theta,t)$ and $A_{3}(\theta,t)$ are uniformly bounded, we use the maximum principle argument as in Lemma \ref{graest} to conclude that $\mathcal{K}_{\theta\theta}$ is uniformly bounded.
\end{proof}
From now on, denoted by $\mathcal{K}_{i}=\frac{\partial^{i}\mathcal{K}}{\partial\theta^{i}}(i=1,2,3\cdots)$. By the induction method, we can get the following conclusion.

\begin{lemma}(Higher order derivatives estimate of $\mathcal{K}$)
For any positive integer $s>2$, if $\mathcal{K},\, \mathcal{K}_{1},\, \mathcal{K}_{2},\cdots, \mathcal{K}_{s}$ are bounded on $S^{1}\times[0,T)$, then $\mathcal{K}_{s+1}$ is uniformly bounded on $S^{1}\times[0,T)$.
\end{lemma}
\begin{proof}
For convenience, we rewrite the equation \eqref{k2} as
\begin{equation}
(\mathcal{K}_{1})_{t}=\frac{\widetilde{p}\mathcal{K}^{2}}{\phi}\mathcal{K}_{3}+\bigg[\big(
\frac{\widetilde{p}\mathcal{K}^{2}}{\phi}\big)_{\theta}+\frac{2\widetilde{p}_{\theta}\mathcal{K}^{2}}{\phi}\bigg]\mathcal{K}_{2}
+\bigg[2\big(\frac{\widetilde{p}_{\theta}\mathcal{K}^{2}}{\phi}\big)_{\theta}+f(\mathcal{K})\bigg]\mathcal{K}_{1}.
\end{equation}
where $f(\mathcal{K})=3\mathcal{K}^{2}-\frac{\lambda(t)\mathcal{K}}{\widetilde{A}}$. We employ the induction method to obtain the equation satisfied by $\mathcal{K}_{s}$,
\begin{align}
(\mathcal{K}_{s})_{t}=&\frac{\widetilde{p}\mathcal{K}^{2}}{\phi}\mathcal{K}_{s+2}+\bigg[s\big(\frac{\widetilde{p}\mathcal{K}^{2}}{\phi}\big)_{\theta}+
\frac{2\widetilde{p}_{\theta}\mathcal{K}^{2}}{\phi}\bigg]\mathcal{K}_{s+1}+\bigg[C^{1}_{s}\big(\frac{\widetilde{p}\mathcal{K}^{2}}{\phi}\big)_{\theta\theta}
+C^{2}_{s}\big(\frac{\widetilde{p}_{\theta}\mathcal{K}^{2}}{\phi}\big)_{\theta}+f(\mathcal{K})\bigg]\mathcal{K}_{s}\notag\\
&+\bigg[C^{1}_{s-1}\big(\frac{\widetilde{p}\mathcal{K}^{2}}{\phi}\big)_{\theta\theta\theta}
+C^{2}_{s-1}\big(\frac{\widetilde{p}_{\theta}\mathcal{K}^{2}}{\phi}\big)_{\theta\theta}+f(\mathcal{K},\mathcal{K}_{1})\bigg]\mathcal{K}_{s-1}+\cdots\cdots\notag\\
&+\bigg[C^{1}_{3}\big(\frac{\partial^{s-1}}{\partial\theta^{s-1}}\frac{\widetilde{p}\mathcal{K}^{2}}{\phi}\big)
+C^{2}_{3}\big(\frac{\partial^{s-2}}{\partial\theta^{s-2}}\frac{\widetilde{p}_{\theta}\mathcal{K}^{2}}{\phi}\big)+f(\mathcal{K},\mathcal{K}_{1},\mathcal{K}_{2},\cdots,\mathcal{K}_{s-3})\bigg]\mathcal{K}_{3}\notag\\
&+\bigg[C^{1}_{2}\big(\frac{\partial^{s}}{\partial\theta^{s}}\frac{\widetilde{p}\mathcal{K}^{2}}{\phi}\big)
+C^{2}_{2}\big(\frac{\partial^{s-1}}{\partial\theta^{s-1}}\frac{\widetilde{p}_{\theta}\mathcal{K}^{2}}{\phi}\big)+f(\mathcal{K},\mathcal{K}_{1},\mathcal{K}_{2},\cdots,\mathcal{K}_{s-2})\bigg]\mathcal{K}_{2}\notag\\
&+\bigg[C^{1}_{1}\big(\frac{\partial^{s}}{\partial\theta^{s}}\frac{\widetilde{p}_{\theta}\mathcal{K}^{2}}{\phi}\big)+f(\mathcal{K},\mathcal{K}_{1},\mathcal{K}_{2},\cdots,\mathcal{K}_{s-1})\bigg]\mathcal{K}_{1},
\end{align}
where $C^{1}_{1}, C^{1}_{2}, C^{2}_{2},\cdots,C^{1}_{s},C^{1}_{s}$ are certain constant coefficients and $f(\mathcal{K},\mathcal{K}_{1},\mathcal{K}_{2},\cdots,\mathcal{K}_{s-1})$ is a polynomial of $\mathcal{K},\mathcal{K}_{1},\mathcal{K}_{2},\cdots,
\mathcal{K}_{s-1}$
 (Similarly for the other terms $f(\mathcal{K},\mathcal{K}_{1},\mathcal{K}_{2},\cdots, \mathcal{K}_{s-2}), \cdots,\\f(\mathcal{K},\mathcal{K}_{1}))$. Since
$\frac{\partial^{s}}{\partial\theta^{s}}\frac{\widetilde{p}\mathcal{K}^{2}}{\phi}$ and $\frac{\partial^{s}}{\partial\theta^{s}}\frac{\widetilde{p}_{\theta}\mathcal{K}^{2}}{\phi}$ contain terms
$\frac{2\widetilde{p}\mathcal{K}\mathcal{K}_{s}}{\phi}$ and $\frac{2\widetilde{p}_{\theta}\mathcal{K}\mathcal{K}_{s}}{\phi}$, respectively, we organize the equation to obtain
\begin{equation*}
(\mathcal{K}_{s})_{t}=\frac{\widetilde{p}\mathcal{K}^{2}}{\phi}\mathcal{K}_{s+2}+\bigg[s\big(\frac{\widetilde{p}\mathcal{K}^{2}}{\phi}\big)_{\theta}+
\frac{2\widetilde{p}_{\theta}\mathcal{K}^{2}}{\phi}\bigg]\mathcal{K}_{s+1}+g_{s}(\mathcal{K},\mathcal{K}_{1},\mathcal{K}_{2})\mathcal{K}_{s}+h_{s}(\mathcal{K},\mathcal{K}_{1},\mathcal{K}_{2},\cdots,\mathcal{K}_{s-1}),
\end{equation*}
where $g_{s}(\mathcal{K},\mathcal{K}_{1},\mathcal{K}_{2})$ is the polynomial given by
$$g(\mathcal{K},\mathcal{K}_{1},\mathcal{K}_{2})=C^{1}_{s}\big(\frac{\widetilde{p}\mathcal{K}^{2}}{\phi}\big)_{\theta\theta}
+C^{2}_{s}\big(\frac{\widetilde{p}_{\theta}\mathcal{K}^{2}}{\phi}\big)_{\theta}+\frac{2C^{1}_{2}\widetilde{p}\mathcal{K}\mathcal{K}_{2}}{\phi}+
\frac{2C^{1}_{1}\widetilde{p}_{\theta}\mathcal{K}\mathcal{K}_{1}}{\phi}+f(\mathcal{K})$$
and $h(\mathcal{K},\mathcal{K}_{1},\mathcal{K}_{2},\cdots,\mathcal{K}_{s-1})$ is the polynomial given by
\begin{align}
h_{s}(\mathcal{K},\mathcal{K}_{1},&\mathcal{K}_{2},\cdots,\mathcal{K}_{s-1})=\bigg[C^{1}_{s-1}\big(\frac{\widetilde{p}\mathcal{K}^{2}}{\phi}\big)_{\theta\theta\theta}
+C^{2}_{s-1}\big(\frac{\widetilde{p}_{\theta}\mathcal{K}^{2}}{\phi}\big)_{\theta\theta}+f(\mathcal{K},\mathcal{K}_{1})\bigg]\mathcal{K}_{s-1}+\cdots\cdots\notag\\
&+\bigg[C^{1}_{3}\big(\frac{\partial^{s-1}}{\partial\theta^{s-1}}\frac{\widetilde{p}\mathcal{K}^{2}}{\phi}\big)
+C^{2}_{3}\big(\frac{\partial^{s-2}}{\partial\theta^{s-2}}\frac{\widetilde{p}_{\theta}\mathcal{K}^{2}}{\phi}\big)+f(\mathcal{K},\mathcal{K}_{1},\mathcal{K}_{2},\cdots,\mathcal{K}_{s-3})\bigg]\mathcal{K}_{3}\notag\\
&+\bigg[C^{1}_{2}\big(\frac{\partial^{s}}{\partial\theta^{s}}\frac{\widetilde{p}\mathcal{K}^{2}}{\phi}-\frac{2\widetilde{p}\mathcal{K}\mathcal{K}_{s}}{\phi}\big)
+C^{2}_{2}\big(\frac{\partial^{s-1}}{\partial\theta^{s-1}}\frac{\widetilde{p}_{\theta}\mathcal{K}^{2}}{\phi}\big)+f(\mathcal{K},\mathcal{K}_{1},\mathcal{K}_{2},\cdots,\mathcal{K}_{s-2})\bigg]\mathcal{K}_{2}\notag\\
&+\bigg[C^{1}_{1}\big(\frac{\partial^{s}}{\partial\theta^{s}}\frac{\widetilde{p}_{\theta}\mathcal{K}^{2}}{\phi}-\frac{2\widetilde{p}_{\theta}\mathcal{K}\mathcal{K}_{s}}{\phi}\big)+f(\mathcal{K},\mathcal{K}_{1},\mathcal{K}_{2},\cdots,\mathcal{K}_{s-1})\bigg]\mathcal{K}_{1}\notag.
\end{align}
Let $F=\mathcal{K}_{s+1}+\alpha\mathcal{K}_{s}^{2}$ with $\alpha>0$ being any fixed positive number. Combining \begin{equation*}
(\mathcal{K}_{s+1})_{t}=\frac{\widetilde{p}\mathcal{K}^{2}}{\phi}\mathcal{K}_{s+3}+\bigg[s\big(\frac{\widetilde{p}\mathcal{K}^{2}}{\phi}\big)_{\theta}+
\frac{2\widetilde{p}_{\theta}\mathcal{K}^{2}}{\phi}\bigg]\mathcal{K}_{s+2}+g_{s+1}(\mathcal{K},\mathcal{K}_{1},\mathcal{K}_{2})\mathcal{K}_{s+1}+h_{s+1}(\mathcal{K},\mathcal{K}_{1},\mathcal{K}_{2},\cdots,\mathcal{K}_{s})
\end{equation*}
we obtain
$$F_{t}=\frac{\widetilde{p}\mathcal{K}^{2}}{\phi}F_{\theta\theta}+\bigg[(s+1)\big(\frac{\widetilde{p}\mathcal{K}^{2}}{\phi}\big)_{\theta}+
\frac{2\widetilde{p}_{\theta}\mathcal{K}^{2}}{\phi}\bigg]F_{\theta}+G(\theta,t)$$
and
\begin{align*}
G(\theta,t)=&\frac{-2\alpha\widetilde{p}\mathcal{K}^{2}}{\phi}\mathcal{K}_{s+1}^{2}+\bigg[g_{s+1}(\mathcal{K},\mathcal{K}_{1},\mathcal{K}_{2})
-2\alpha\big(\frac{\widetilde{p}\mathcal{K}^{2}}{\phi}\big)_{\theta}\mathcal{K}_{s}\bigg]\mathcal{K}_{s+1}
+h_{s+1}(\mathcal{K},\mathcal{K}_{1},\mathcal{K}_{2},\cdots,\mathcal{K}_{s})\\
&+2\alpha\big[g_{s}(\mathcal{K},\mathcal{K}_{1},\mathcal{K}_{2})\mathcal{K}_{s}
+h_{s}(\mathcal{K},\mathcal{K}_{1},\mathcal{K}_{2},\cdots,\mathcal{K}_{s-1})\big]\mathcal{K}_{s}
\end{align*}
For any $\alpha>0$, the coefficient of $\mathcal{K}_{s+1}^{2}$ is negative. Still by the argument in Lemma \ref{hies}, we can get the estimates for $F$ and thus for $\mathcal{K}_{s+1}$.
\end{proof}
By the above curvature estimates, we can conclude that the flow exists globally.
\begin{lemma}\label{glob}
(Global existence of the flow) The flow \eqref{flow1} exists in the time interval $[0,\infty)$.
\end{lemma}
\section{Convergence of the flow}
We first show the Hausdorff convergence of the flow \eqref{flow1}.

\begin{lemma}
Under the flow \eqref{flow1}, the anisoperimetric deficit $\frac{\mathcal{L}^{2}}{4A\widetilde{A}}-1$ is nonincreasing in time and converges to zero exponentially as $t\to\infty$.
\end{lemma}
\begin{proof}
From the evolution of $\mathcal{L}$ and $A$, we compute to obtain
\begin{align*}
\frac{d}{dt}\big(\frac{\mathcal{L}^{2}}{4A\widetilde{A}}-1\big)=&\frac{-\mathcal{L}^{2}A_{t}}{4\widetilde{A}A^{2}}
=\frac{-\mathcal{L}^{2}\big(-2\widetilde{A}+\frac{\mathcal{L}}{2\widetilde{A}}\int_{0}^{L}\widetilde{p}\mathcal{K}^{2}ds\big)}
{4\widetilde{A}A^{2}}
\le\frac{2\mathcal{L}^{2}\widetilde{A}-\frac{\mathcal{L}^{4}}{2A}}{4\widetilde{A}A^{2}}
=-\frac{\mathcal{L}^{2}}{2A^{2}}(\frac{\mathcal{L}^{2}}{4A\widetilde{A}}-1\big).
\end{align*}
Since $A(t)\le\frac{\mathcal{L}^{2}(t)}{4\widetilde{A}}$, we have
$$\frac{d}{dt}\big(\frac{\mathcal{L}^{2}}{4A\widetilde{A}}-1\big)\le
-\frac{\mathcal{L}^{2}}{2\big(\frac{\mathcal{L}^{2}}{4\widetilde{A}}\big)^{2}}(\frac{\mathcal{L}^{2}}{4A\widetilde{A}}-1\big)
=-\frac{8\widetilde{A}^{2}}{\mathcal{L}^{2}}(\frac{\mathcal{L}^{2}}{4A\widetilde{A}}-1\big).$$
Integrating this yields
$$\frac{\mathcal{L}^{2}}{4A\widetilde{A}}-1\le\big(\frac{\mathcal{L}^{2}(0)}{4\widetilde{A}A(0)}-1\big)e^{
-\frac{8\widetilde{A}^{2}}{\mathcal{L}^{2}(0)}t}.$$
\end{proof}
Denote by $K$ the convex domain enclosed by convex curve $X$. Let
$$r_{in}=\max\{r>0|r\widetilde{X}+x\subseteq K,\,x\in K\}$$
and
$$r_{out}=\min\{r>0|r\widetilde{X}+x\supseteq K,\,x\in K\}.$$
According to \cite{Gage93}, it holds that
\begin{equation}
\frac{\mathcal{L}^{2}}{4A\widetilde{A}}-1\ge\frac{\widetilde{A}}{4A}(r_{out}-r_{in})^{2}  \quad\text{(Bonnesen inequality)},
\end{equation}
where the equality holds if and only if $X$ is a homothety of the boundary of Wulff shape $\widetilde{X}$. Then we have the following corollary.
\begin{corollary}\label{haus}
(Hausdorff convergence) The flow \eqref{flow1} converges to the boundary of Wulff shape $\widetilde{X}$ in the sense of Hausdorff metric as $t\to\infty$.
\end{corollary}

\begin{lemma}
Under the flow \eqref{flow1}, we have
\begin{equation}\label{at0}
\frac{dA}{dt}(t)\to 0\quad \text{as} \to\infty.
\end{equation}
\end{lemma}
\begin{proof}
Assume not for \eqref{at0}. Since $\frac{dA}{dt}(t)\ge0$, there exists a constant $D_{0}>0$ independent of time and a sequence of times $\{t_{i}\}^{\infty}_{i=1}$ going to infinity such that
$$\frac{dA}{dt}(t_{i})\ge D_{0}>0$$
for all $i=1,2,3\cdots$. Also note that
$$\bigg\lvert\frac{d^{2}A}{dt^{2}}(t)\bigg\lvert=\bigg\lvert\frac{d}{dt}\big(
-2\widetilde{A}+\frac{\mathcal{L}}{2\widetilde{A}}\int_{0}^{L}\widetilde{p}\mathcal{K}^{2}ds\big)\bigg\lvert
=\bigg\lvert\frac{\mathcal{L}}{2\widetilde{A}}\int_{0}^{2\pi}\frac{d}{dt}\big(\frac{\widetilde{p}\phi^{2}\kappa^{2}}{\kappa}\big)d\theta\bigg\lvert
\le\frac{\mathcal{L}\max_{S^{1}}(\widetilde{p}\phi^{2})}{2\widetilde{A}}\int_{0}^{2\pi}\lvert\kappa_{t}\lvert d\theta.$$
By the equation \eqref{kt} and the estimates established earlier, there exists a constant $D_{1}$ such that $\big\lvert\frac{d^{2}A}{dt^{2}}(t)\big\lvert\le D_{1}$ for all $t\in[0,\infty)$. As the slope of $A^{'}(t)$ is uniformly bounded, one can find a number $\mu_{0}$ independent of $t_{i}$ such that
$$A^{'}(t)\ge\frac{D_{0}}{2}>0,\quad\forall t\in[t_{i},t_{i}+\mu_{0}].$$
This implies that
$$A(\infty)-A(0)=\int^{\infty}_{0}A^{'}(t)dt=\infty,$$
which contradicts the inequality $0<A(t)\le\frac{\mathcal{L}^{2}(t)}{4\widetilde{A}}=\frac{\mathcal{L}^{2}(0)}{4\widetilde{A}}$. This finishes the proof.
\end{proof}
Next we prove the convergence of anisotropic curvature.

\begin{lemma}\label{cwu}
Under the flow \eqref{flow1}, we have
\begin{equation}
\lim\limits_{t\to\infty}\bigg\Vert\mathcal{K}(\theta,t)-\frac{2\widetilde{A}}{\mathcal{L}(0)}\bigg\Vert_{C^{m}(S^{1})}=0,
\quad \forall m=0,1,2,3\cdots
\end{equation}
\end{lemma}
\begin{proof}
By the regularity estimate, it is sufficient to show $C^{0}$ convergence. By the well-known Arzela-Ascoli theorem, for any sequence of time $\{t_{i}\}^{\infty}_{i=1}$ going to infinity, there is a subsequence (still denoted by $\{t_{i}\}^{\infty}_{i=1}$) such that $\mathcal{K}(\theta,t_{i})\to\mathcal{K}_{\infty}(\theta)$ and $\kappa
(\theta,t_{i})\to\kappa_{\infty}(\theta)$ in $C^{\infty}(S^{1})$ as $i\to\infty$, where $\mathcal{K}_{\infty}=\phi\kappa_{\infty}=\frac{\kappa_{\infty}}{\widetilde{\kappa}}$. Moreover,
\begin{align*}
\frac{dA}{dt}&=-\int_{0}^{L}\widetilde{p}\mathcal{K}ds+\frac{\int_{0}^{L}\widetilde{p}ds}{2\widetilde{A}}\int_{0}^{L}\widetilde{p}\mathcal{K}^{2}ds\\
&=-\int_{0}^{L}\widetilde{p}\mathcal{K}ds+\frac{\int_{0}^{L}\widetilde{p}ds}{\int_{0}^{L}\widetilde{p}\mathcal{K}ds}\int_{0}^{L}\widetilde{p}\mathcal{K}^{2}ds\\
&\to-\int_{0}^{L}\widetilde{p}\mathcal{K}_{\infty}ds+\frac{\int_{0}^{L}\widetilde{p}ds}{\int_{0}^{L}\widetilde{p}\mathcal{K}_{\infty}ds}\int_{0}^{L}\widetilde{p}\mathcal{K}_{\infty}^{2}ds=0
\end{align*}
Hence
$$\int_{0}^{L}\widetilde{p}ds\int_{0}^{L}\widetilde{p}\mathcal{K}_{\infty}^{2}ds=\big(\int_{0}^{L}\widetilde{p}\mathcal{K}_{\infty}ds\big)^{2}.$$
The equality of Cauchy-Schwarz Inequality tells us that $\mathcal{K}_{\infty}\equiv C^{\ast}$,where $C^{\ast}$ is a constant. In particular, we have
$$\mathcal{K}(\cdot,t_{i})\to C^{\ast}\,\,\text{in}\,\,C^{\infty}(S^{1})\,\,\text{as}\,\,i\to\infty,$$
or equivalently
$$\kappa(\cdot,t_{i})\to C^{\ast}\widetilde{\kappa}\,\,\text{in}\,\,C^{\infty}(S^{1})\,\,\text{as}\,\,i\to\infty.$$
Since the anisotropic length is preserving, i.e., $\mathcal{L}(t)=\mathcal{L}(0)$, combining with \eqref{a}, \eqref{anl} and \eqref{wua} yields
$$C^{\ast}=\frac{2\widetilde{A}}{\mathcal{L}(0)}.$$
As a result, we conclude that $\mathcal{K}(\cdot,t)$ converges to $\frac{2\widetilde{A}}{\mathcal{L}(0)}$ as $t\to\infty$.
\end{proof}

The proof of Theorem \ref{them} is now complete due to Lemma \ref{ala}, Lemma \ref{conv}, Lemma \ref{glob}, Remark \ref{haus} and Lemma \ref{cwu}.

\textbf{Acknowledgements}
This work was partially supported by Science and Technology Commission
of Shanghai Municipality (No. 22DZ2229014). The research is supported by Shanghai Key Laboratory of PMMP.

\textbf{Data availability}
Data sharing  not applicable to this article as no datasets were generated or analysed during the current study.

\textbf{Conflict of interest}
The author has no conflicts of interest to declare that are relevant to the content of this article.


\end{document}